\documentclass{article}
\usepackage{amsmath}
\usepackage{amssymb,amsfonts}
\usepackage{t1enc}
\usepackage{graphicx}
\usepackage{mathptmx}
\usepackage{color}
\usepackage{comment}
\usepackage[utf8]{inputenc}
\usepackage{geometry}
\usepackage{amsthm}
\usepackage{mathdots}
\usepackage[numbers,sort&compress]{natbib}
\usepackage{hyperref}
\usepackage[all]{xy}

\numberwithin{equation}{section}

\newcommand{\Cm}{\ensuremath{\operatorname{{\mathfrak{Cm}}}}}   
\newcommand{\Cf}{\ensuremath{\operatorname{{\mathfrak{Cf}}}}}   

\newcommand{\df}{\ensuremath{\overset{\mathrm{df}}{=}}}

\newcommand{\tiff}{if and only if \ }
\newcommand{\Iff}{\Longleftrightarrow}

\newcommand{\Implies}{\Rightarrow}
\newcommand{\klam}[1]{\ensuremath{\langle #1 \rangle}}
\newcommand{\set}[1]{\ensuremath{\{#1\}}}
\newcommand{\z}{\emptyset}
\newcommand{\tand}{\text{ and }}
\newcommand{\tor}{\text{ or }}

\newcommand{\nec}[1]{\ensuremath{[ #1 ]}}
\newcommand{\poss}[1]{\ensuremath{\klam{ #1 }}}

\newcommand{\lr}{\nec{\leq_1}\poss{\leq_2}}

\newcommand{\F}{\mathcal{F}}
\newcommand{\I}{\mathcal{I}}

\newcommand{\wlg}{w.l.o.g.~}

\newcommand{\ua}{\uparrow}
\newcommand{\uaE}{\uparrow_{\leq_1}}
\newcommand{\uaZ}{\uparrow_{\leq_2}}
\newcommand{\da}{\downarrow}

\newcommand{\onto}{\twoheadrightarrow}

\theoremstyle{plain}
\newtheorem{theorem}{Theorem}[section]
\newtheorem{lemma}[theorem]{Lemma}

\theoremstyle{definition}

\theoremstyle{remark}

\parskip=\medskipamount
\date{}

\title{Representation of lattice frames}

\author{
Ivo D\"untsch%
\thanks{Visiting professor at the School of Mathematics \& Computer Science, Fujian Normal University, Fuzhou, Fujian, China} \ ${}^{,}$
\thanks{Ivo D\"untsch gratefully acknowledges support by the Natural Sciences and Engineering Research Council of Canada Discovery Grant 250153 and  by the Bulgarian National Fund of Science, contract DN02/15/19.12.2016.}
 \\
{Dept of Computer Science} \\
{Brock University} \\
{St Catharines, ON, L2S 3A1, Canada} \\
\href{mailto:duentsch@brocku.ca}{duentsch@brocku.ca}
\and
Ewa Or{\l}owska%
\\
National Institute of Telecommunications \\
Szachowa 1 \\
04--894, Warszawa, Poland \\
\href{mailto:orlowska@itl.waw.pl}{orlowska@itl.waw.pl}
}

\begin{document}
\maketitle

\begin{abstract}
\noindent The aim of this note is to characterize those doubly ordered frames $X$ which are embeddable into the canonical frame of its Urquhart complex algebra.
\end{abstract}

\section{Urquhart's lattice representation}

Throughout, $\klam{L, +, \cdot 0, 1}$ is a bounded lattice. If no confusion can arise we shall identify algebras with their base set. The collection of proper filters of $L$ is denoted by $\F$, and the collection of proper ideals of $L$ is denoted by $\I$. A \emph{filter -- ideal pair} is a pair $\klam{F,I}$ where $F \in \F,\ I \in \I$, and $F \cap I = \z$. A filter -- ideal pair $\klam{F,I}$ is called \emph{maximal}, if $F$ is maximally disjoint to $I$ and $I$ is maximally disjoint to $F$. In other words, if $F' \in \F$ such that $F \subsetneq F'$, then $F' \cap I \neq \z$, and if $I' \in \I$ such that $I \subsetneq I'$, then $F \cap I' \neq \z$.

Let $X_L$ be the set of all maximal filter -- ideal pairs. To facilitate notation, if $x \in X_L$ with $x = \klam{F,I}$ we let $x_1 = F$ and $x_2 = I$. We define two relations $\leq_1, \leq_2$ on $X_L$ by $x \leq_i y$ \tiff $x_i \subseteq y_i$. Clearly, $x_1$ and $x_2$ are quasiorders on $X_L$. The structure $\klam{X_L, \leq_1, \leq_2}$ is called the \emph{Urquhart canonical frame of $L$}, denoted by $\Cf_U(L)$.

\begin{lemma}\label{lem:ext}
Each filter -- ideal pair can be extended to a maximal pair.
\end{lemma}

A \emph{doubly ordered frame} is a structure $\klam{X, \leq_1, \leq_2}$ such that
\begin{enumerate}
\item $\leq_1$ and $\leq_2$ are quasiorders on $X$.
\item If $x \leq_1 y$ and $x \leq_2 y$, then $x = y$ for all $x,y \in X$.
\end{enumerate}

For $Y \subseteq X$,
\begin{align}
l(Y) &\df \set{x: \ua_1 x \cap Y = \z} = \nec{\leq_1}(-Y), \label{def:l}\\
r(Y) & \df \set{x: \ua_2 x \cap Y = \z} = \nec{\leq_2}(-Y). \label{def:r}
\end{align}

$Y$ is called a \emph{stable set}, if $Y = l(r(Y))$. The collection of stable sets is denoted by $L_X$. Observe that
\begin{gather}
l(r(Y)) = l(\nec{\leq_2}(-Y)) = \nec{\leq_1}(-\nec{\leq_2}(-Y)) = \nec{\leq_1}\poss{\leq_2}(Y)
\label{lr}.
\end{gather}

\begin{lemma} \label{DOS:3} \cite{urq78} Let $(X, \leq_1, \leq_2)$ be a doubly ordered frame.
\begin{enumerate}
\item The mappings $l$ and $r$ form a Galois connection between the lattice of $\leq_1$--increasing subsets of $X$ and the lattice of $\leq_2$--increasing subsets of $X$.
\item If $Y$ is $\leq_2$ increasing, then $l(Y)$ is a stable set.
    \end{enumerate}
\end{lemma}
Thus, if $Y$ is $\leq_1$ increasing and $Z$ is $\leq_2$ increasing, then $Y \subseteq l(Z)$  \tiff  $Z\subseteq r(Y)$.

For $Y,Z \in L_X$ let
\begin{align}
Y \lor_{X} Z &\df \nec{\leq_1}\poss{\leq_2}(Y \cup Z), \label{compsup} \\
Y \land_{X} Z &\df Y \cap Z. \label{compinf}
\end{align}
\begin{theorem}\label{thm:repL} \cite{urq78}
The structure $\klam{L_X, \lor_{X}, \land_{X}, \z, X}$ is a complete bounded lattice.
\end{theorem}

We call this structure the \emph{Urquhart complex algebra of $X$}, and denote it by $\Cm_U(X)$.

\begin{theorem}\label{thm:urq} \cite{urq78}
Define $h: L \to 2^{X_L}$ by $h(a) \df \set{x \in X_L: a \in x_1}$. Then $h$ is a lattice embedding into $\Cm_U\Cf_U(L)$.
\end{theorem}

It was shown by \citet{ch14} that $\Cm_U\Cf_U(L)$ is isomorphic to the canonical extension of $L$ in the sense of \cite{gh01}.

\section{Representability of lattice frames}

\citet{urq78} proved that every doubly ordered frame endowed with a suitable topology can be embedded into the dual frame of its dual lattice. We show  below on a first order level that his conditions suffice to prove that a suitably defined lattice frame can be embedded into the canonical frame of its complex algebra.

A \emph{lattice frame} is a doubly ordered frame $\klam{X, \leq_1, \leq_2}$ which satisfies the following conditions:
\begin{enumerate}
\renewcommand{\theenumi}{\ensuremath{\mathrm{LF}_\arabic{enumi}}}
\item Each element of $X$ is below a $\leq_1$ maximal one and a $\leq_2$ maximal one, \label{fr0}
\item $x \not\leq_1 y \Implies (\exists z)[y \leq_1 z \tand (\forall w)(x \leq_1 w \Implies z \not\leq_2 w)]$, \label{fr1}
\item $x \not\leq_2 y \Implies (\exists z)[y \leq_2 z \tand (\forall w)(x \leq_2 w \Implies z \not\leq_1 w)]$. \label{fr2}
\end{enumerate}


\ref{fr1} and \ref{fr2} are the conditions given by \citet{urq78} for lattices of finite length. In such lattices, they guarantee embeddability of $X$ into $\Cf\Cm_U(X)$. They hold in all canonical frames:

\begin{theorem}\label{thm:latframe1} \cite{urq78}
If $L$ is a lattice, then $X_L$ is a lattice frame. 
\end{theorem}
\begin{proof}
\ref{fr0}: By Zorn's Lemma, each filter (ideal) is contained in a maximal one.

\ref{fr1} Assume that \ref{fr1} is not true. Then,
\begin{gather}\label{fr1eq1}
(\exists x,y)[x \not\leq_1 y \tand (\forall z)(y \leq_1 z \Implies (\exists w)(x \leq_1 w \tand z \leq_2 w))].
\end{gather}

Let $x,y \in X$ witness \eqref{fr1eq1}. Since $x_1 \not\subseteq y_1$, there is some $a \in x_1, a \not\in y_1$. Thus, $\da_1 a \cap y_1 = \z$, and so there is a maximal pair $z$ such that $y_1 \subseteq z_1$ and $a \in z_2$. The assumption \eqref{fr1eq1} implies that there is a maximal pair $w$ such that $x_1 \subseteq w_1$ and $z_2 \subseteq w_2$. Since $w$ is a maximal pair, $w_1 \cap w_2 = \z$ which contradicts $a \in x_1 \cap z_2$.

\ref{fr2}: This is shown similarly: Assume that \ref{fr2} is not true. Then,
\begin{gather}\label{fr1eq2}
(\exists x,y)[x \not\leq_2 y \tand (\forall z)(y \leq_2 z \Implies (\exists w)(x \leq_2 w \tand z \leq_1 w))].
\end{gather}
Since $x_2 \not\subseteq y_2$, there is some $a \in x_2, a \not\in y_2$. Thus, $\ua_2 a \cap y_2 = \z$, and so there is a maximal pair $z$ such that $y_2 \subseteq z_2$ and $a \in z_1$. The assumption \eqref{fr1eq2} implies that there is a maximal pair $w$ such that $x_2 \subseteq w_2$ and $z_1 \subseteq w_1$. Since $w$ is a maximal pair, $w_1 \cap w_2 = \z$ which contradicts $a \in x_2 \cap z_1$.
\end{proof}


\begin{theorem}\label{thm:ep}
Let $X$ be a lattice frame. Then, $X$ is embeddable into $\Cf_U\Cm_U(X)$.
\end{theorem}
\begin{proof}
Let $L_X$ be the lattice of the stable sets of $X$. Define $k_1, k_2: X \to 2^{L_X}$ by $k_1(x) = \set{Y \in L_X: x \in Y}$, $k_2(x) = \set{Y \in L_X: x \in r(Y)}$, and $k(x) = \klam{k_1(x), k_2(x)}$.

We shall show that
\begin{enumerate}
\item $k$ preserves $\leq_1$ and $\leq_2$.
\item $k$ is injective.
\item $k(x)$ is a maximal pair of $\Cm_U(X)$.
\end{enumerate}
We first show that $k_1(x)$ is the principal filter $F_x$ of $L_X$ generated by $\ua_1(x)$. If $Y \in k_1(x)$, then  $Y \in L_X$ and $x \in Y$. Since $lr$ is a closure operator on the $\leq_1$ -- closed sets, $lr(\ua_{\leq_1} x) \subseteq lr(Y) =  Y$, and all that is left to show is that $lr(\ua_{\leq_1} x) \subseteq \ua_{\leq_1}$. Consider
\begin{align}
y \in lr(\ua_{\leq_1} x) &\Iff y \in \nec{\leq_{\leq_1}}\poss{\leq_2}\ua_{\leq_1} x , \\
&\Iff \ua_{\leq_1} y \subseteq \poss{\leq_2}\ua_{\leq_1} x, \\
&\Iff (\forall z)[y \leq_{\leq_1} z \Implies (\exists t)(x \leq_{\leq_1} t \tand z \leq_2 t)], \\
&\Iff (\forall z)[y \leq_{\leq_1} z \Implies \ua_{\leq_1} x \cap \ua_{\leq_2} z \neq \z]. \label{t1}
\end{align}
Let $y \in lr(\ua_{\leq_1} x)$ and assume that $x \not\leq_1 y$. By \ref{fr1}, there is some $z$ such that $y \leq_1 z$ and $\ua_{\leq_1} x \cap \ua_{\leq_2} z = \z$. This contradicts \eqref{t1}.


Preservation of $\leq_1$ and $\leq_2$ is immediate. For injectivity, let $x \neq y$ and assume $k(x) = k(y)$, i.e. $k_1(x) = k_1(y)$ and $k_2(x) = k_2(y)$. Then, $k_1(x) = k_1(y)$ implies $\ua_1 x = \ua_1 y$, i.e. $x \leq_1 y$ and $y \leq_1 x$. Since $X$ is doubly ordered we may suppose \wlg that $x \not\leq_2 y$. By \ref{fr2} there is some $z$ such that $y \leq_2 z$ and $\ua_{\leq_2} x ~ \cap \ua_{\leq_1} z = \z$. Then, $x \in r(\ua_{\leq_1} z)$ and $y \not\in \ua_{\leq_1} z$, contradicting $k_2(x) = k_2(y)$.

Clearly, $k_1(x)$ is a filter of $L_X$, $k_2(x)$ is an ideal, and $k_1(x) \cap k_2(x) = \z$. All that is left to show is that $k(x)$ is a maximal pair. Assume that $F$ is a filter of $L_X$ strictly containing $k_1(x)$ and $F \cap k_2(x) = \z$. Let $Y \in F \setminus k_1(x)$. Since $\ua_1 x \in k_1(x)$ and $F$ is a filter, it follows that $Z \df Y \cap \ua_1 x \in F$ and $Z \subseteq \ua_1 x$. Then, $t \in Z$ implies $x \lneq_1 t$, and thus, $x \not\leq_2 t$ for all $t \in Z$. By the assumption we have $Z \not\in k_2(x)$, and thus, $x \not\in r(Z)$. Hence, $x \not\in \nec{\leq_2}(-Z)$, and there is some $z$ such that $x \leq_2 z$ and $z \in Z$. This contradicts $x \not\leq_2 t$ for all $t \in Z$, and thus, $k_1(x)$ is maximally disjoint from $k_2(x)$.

Finally, we show that $k_2(x)$ is maximally disjoint from $k_1(x)$. Let $W_x \df \set{y: x\not\leq_2 y}$; clearly, $\uaZ x \cap W_x = \z$. If $Y \subseteq X$ such that $x \in r(Y)$, then $\uaZ x \cap Y = \z$, and thus, $Y \subseteq W_x$. Therefore, $W_x$ is the largest subset of $X$ disjoint from $\uaZ x$, and, clearly, $W_x = \poss{\leq_2}(W_x)$. It follows that $\lr(W_x) = \nec{\leq_1}(W_x)$ is the largest stable set $Y$ for which $x \in r(Y)$. Hence, $k_2(x)$ is the ideal of $L_X$ generated by $\nec{\leq_1}(W_x)$.

Suppose that $I$ is an ideal of $L_X$ which strictly contains $k_2(x)$. Our aim is to show that $I \cap k_1(x) \neq \z$, in other words, there is some $Y \in I$ such that $x \in Y$. If $Y \in I$ and $Y \not\subseteq \nec{\leq_1}(W_x)$, there is some $t \in Y \setminus \nec{\leq_1}(W_x)$. Since $I$ is an ideal, $Y \in I$ and $\uaE t \subseteq Y$, we have $\uaE t \in I$, and therefore $\nec{\leq_1}(W_x) \lor_X \uaE t = \lr(\nec{\leq_1}(W_x) \cup \uaE t) \in I$. Since $t \not\in \nec{\leq_1}(W_x)$,  there is some $s$ such that $t \leq_1 s$ and $x \leq_2 s$. Now,
\begin{align*}
x \in \lr(\nec{\leq_1}(W_x) \cup \uaE t) & \Iff \uaE x \subseteq \poss{\leq_2}(\nec{\leq_1}(W_x) \cup \uaE t) \\
&\Iff (\forall y)[x \leq_1 y \Implies (\exists z)(y \leq_2 z \tand z \in (\nec{\leq_1}(W_x) \cup \uaE t))] \\
&\Iff (\forall y)[x \leq_1 y \Implies (\exists z)(y \leq_2 z \tand [(\forall u)(z \leq_1 u \Implies x\not\leq_2 u) \tor t \leq_1 z])], \\
&\Iff (\forall y)[x \leq_1 y \Implies (\exists z)(y \leq_2 z \tand [\uaE z ~\cap \uaZ x = \z \tor t \leq_1 z])]
\end{align*}
For the right hand side, we consider two cases:
\begin{enumerate}
\item $x = y$: Then, setting $z = s$, we obtain $x \leq_2 z$ and $t \leq_1 z$.
\item $x \lneq_1 y$: Then, $x \not\leq_2 y$, and \ref{fr2} implies that there is some $z$ such that $\uaZ x \cap \uaE z = \z$.
\end{enumerate}
Thus, the right hand side is fulfilled for all $x \leq_1 y$, and it follows that $x \in \lr(\nec{\leq_1}(W_x) \cup \uaE t) \in I$. Hence, $I \cap k_1(x) \neq \z$.
\end{proof}

\section{Modal definability of doubly ordered frames}

If $F = \klam{X, R_1, \ldots, R_n}$ and $F' = \klam{X', R_1' \ldots, R_n'}$ are binary frames, a mapping $f: X \to X'$ is a \emph{bounded morphism} if
\begin{quote}
\begin{enumerate}
\renewcommand{\theenumi}{\ensuremath{\mathrm{BM}_\arabic{enumi}}}
\item\label{bm1} $xR_iy$ implies $f(x)R_i'f(y)$ for all $1 \leq i \leq n$ and $x,y \in X$.
\item\label{bm2} If $f(x)R_i'y'$, then there exists some $y \in X$ such that $xR_iy$ and $f(y) = y'$.
\end{enumerate}
\end{quote}

\begin{theorem}
The class of doubly ordered frames is not modally definable.
\end{theorem}
\begin{proof}
By the Goldblatt -- Thomason Theorem \cite{gt74} it is enough to show that the class is not closed under bounded morphisms. Let $\F = \klam{X,R_1,R_2}$ be a frame such that $X = \set{x,y,z}$, $R_1 = 1' \cup \set{\klam{x,y}}$, and $R_2 = 1' \cup \set{\klam{x,z}}$; then, $\F$ is a doubly ordered frame. Next, let $\F' = \klam{Y,S_1,S_2}$, where $Y = \set{s,t}$, $S_1 = 1' \cup \set{\klam{s,t}}$, and $S_2 \df S_1$; observe that $\F'$ is not doubly ordered.

 Let $f: X \onto Y$ be defined by $f(x) = s, f(y) = f(z) = t$. Clearly, $f$ preserves $R_1$ and $R_2$, and thus, it satisfies \ref{bm1}.

For \ref{bm2}, let $f(u)S_1v$. We need to find some $w \in X$ such that $uR_1w$ and $f(w) = v$. If $v = s$, then $x = u$. If $v = s$, then set $w = x$, if $v = t$, then set $w = y$. If $f(u) = t$, then $v = t$, and the reflexivity of $R_1$ gives the result. For $R_2$ the procedure is analogous, using $z$ instead of $y$.

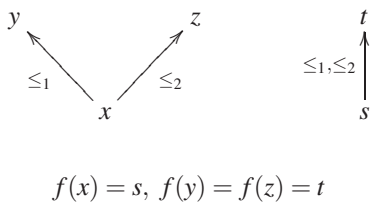
\begin{figure}[htb]
\caption{Doubly ordered frames are not definable by modal operators}\label{fig1}
$$ \xymatrix{
y&&z && t\\
& x\ar@{->}[lu]^{\leq_1} \ar@{->}[ru]_{\leq_2} &&& s \ar@{->}[u]^{\leq_1, \leq_2}
}
$$
\begin{gather*}
f(x) = s, \ f(y) = f(z) = t
\end{gather*}
\end{figure}

Thus $\F'$ is a bounded image of a doubly ordered frame. On the other hand, $s \neq t$ implies that $\F'$ is not doubly ordered.
\end{proof}

\section*{References}
\renewcommand*{\refname}{}


\vspace{-7mm}

\begin{thebibliography}{}

\bibitem[Craig and Haviar, 2014]{ch14}
Craig, A. and Haviar, M. (2014).
\newblock Reconciliation of approaches to the construction of canonical
  extensions of bounded lattices.
\newblock {\em Mathematica Slovaka}, 6:1335--1356.

\bibitem[Gehrke and Harding, 2001]{gh01}
Gehrke, M. and Harding, J. (2001).
\newblock Bounded lattice expansions.
\newblock {\em Journal of Algebra}, 238:345--371.

\bibitem[Goldblatt and Thomason, 1974]{gt74}
Goldblatt, R. and Thomason, S. (1974).
\newblock Axiomatic classes in propositional modal logic.
\newblock In Crossley, J., editor, {\em Algebra and Logic: {P}apers from the
  1974 Summer Research Institute of the {A}ustralian {M}athematical {S}ociety},
  volume 450 of {\em Lecture Notes in Mathematics}, pages 163--173.
  Springer-Verlag, Heidelberg.

\bibitem[Urquhart, 1978]{urq78}
Urquhart, A. (1978).
\newblock A topological representation theorem for lattices.
\newblock {\em Algebra Universalis}, 8:45--58.

\end{thebibliography}
\end{document}